\documentclass[twoside, 11pt]{article}

\usepackage{subfig}
\usepackage{amsmath, wasysym, amssymb}
\usepackage{amsfonts}
\usepackage{verbatim}
\usepackage{enumerate}
\usepackage{color}
\usepackage[sort&compress,round,semicolon,authoryear]{natbib}
\usepackage[colorlinks,citecolor=blue,urlcolor=blue]{hyperref}
\usepackage{chngcntr}
\usepackage{natbib}
\usepackage[amsthm, thmmarks, amsmath]{ntheorem}

\newtheorem{thm}{Theorem}[section]

\newtheorem{lem}[thm]{Lemma}
\newtheorem{prop}[thm]{Proposition}

\theoremstyle{definition}
\newtheorem{defn}[thm]{Definition}

\theoremstyle{remark}

\newtheorem{exmpl}[thm]{Example}

\newcommand{\norm}[1]{\left\Vert#1\right\Vert}

\newcommand{\abs}[1]{\left\vert#1\right\vert}

\newcommand{\R}{\mathbb R}

\newcommand{\Z}{\mathbb Z}

\newcommand{\To}{\rightarrow}

\DeclareMathOperator{\Epi}{epi}

\DeclareMathOperator{\Loc}{loc}

\DeclareMathOperator{\Lin}{Lin}
\DeclareMathOperator{\Span}{span}

\DeclareMathOperator{\Co}{co}

\DeclareMathOperator{\Cone}{cone}

\def\Hy{Hyv\"arinen}

\newcommand\blfootnote[1]{%
  \begingroup
  \renewcommand\thefootnote{}\footnote{#1}%
  \addtocounter{footnote}{-1}%
  \endgroup
}

\providecommand{\keywords}[1]{\textbf{\textit{Keywords:}} #1}

\title{Existence and Uniqueness of Proper Scoring Rules}
\author{Evgeni Y. Ovcharov
\thanks{The author has been supported by the European Union Seventh Framework Programme under grant agreement no. 290976.}
\thanks{Much of this work was done while the author was a PostDoc at the University of Heidelberg, Germany.}
\\
\small Heidelberg Institute for Theoretical Studies\\
\small  Schloss-Wolfsbrunnenweg 35, D-69118 Heidelberg, Germany\\
\small \texttt{trulr6@yahoo.com}
}

\begin{document}

\maketitle

\begin{abstract}
To discuss the existence and uniqueness of proper scoring rules one needs to extend the associated entropy functions as sublinear functions to the conic hull of the prediction set. In some natural function spaces, such as the Lebesgue $L^p$-spaces over $\R^d$, the positive cones have empty interior. Entropy functions defined on such cones have only directional derivatives. Certain entropies may be further extended continuously to open cones in normed spaces containing signed densities. The extended densities are G\^ateaux differentiable except on a negligible set and have everywhere continuous subgradients due to the supporting hyperplane theorem. We introduce the necessary framework from analysis and algebra that allows us to give an affirmative answer to the titular question of the paper. As a result of this, we give a formal sense in which entropy functions have uniquely associated proper scoring rules. We illustrate our framework by studying the derivatives and subgradients of the following three prototypical entropies: Shannon entropy, Hyv\"arinen entropy, and quadratic entropy.
\end{abstract}


.

\pagestyle{myheadings}
\markboth{Ovcharov}{Existence and Uniqueness of Proper Scoring Rules}

\section{Introduction}

Proper scoring rules have attracted a lot of interest in recent years in disparate fields such as statistics, decision theory, machine learning, game theory, finance, meteorology, etc. They provide practical measures for assessing the accuracy and precision of probabilistic forecasts. In this paper, we build a general measure-theoretic framework for proper scoring rules that allows us to consider their existence and uniqueness as subgradients of sublinear functions.

\subsection{Definitions}
Let $(\Omega,\mathcal A,\mu)$ be a measure space and $\mathcal P$ be a convex set of probability densities on $\Omega$ with respect to the measure $\mu$. A random variable $X$ takes values in $\Omega$ with unknown true density $p\in\mathcal P$. We refer to $\mathcal P$ and its elements as a \emph{prediction set} and \emph{predictive densities} for $X$, respectively. By $\mathcal L(\mathcal P)$ we denote the set of all $\mu$-measurable functions $f:\Omega\To\R$ such that
\begin{equation*}
  \int_\Omega \abs{f(x)}p(x)d\mu(x) < \infty
\end{equation*}
for all $p\in\mathcal P$. We call the elements of $\mathcal L(\mathcal P)$ $\mathcal P$-\emph{integrable} functions.

A \emph{scoring rule} $S:\mathcal P\To\mathcal L(\mathcal P)$ assigns for each predictive density $q\in\mathcal P$ a $\mathcal P$-integrable function $S(q)$. The value of $S(q)$ at $x\in\Omega$ is interpreted as a numerical score assigned to the outcome $x$. We take scoring rules to be \emph{positively orientated}, that is, they are viewed as incentives which a forecaster wishes to maximise. It is customary to term $S$ \emph{proper} if the expected value of $S$ at $q$,
\begin{equation*}
  p\cdot S(q):=\int_\Omega S(q)(x)p(x)d\mu(x),
\end{equation*}
is maximised in $q$ at the true density $q=p$, and \emph{strictly proper}, if the true density is the only maximiser.

Strictly proper scoring rules could be used as a bonus system under which truth-telling is the only optimal long-term strategy \citep{GR}. For such an $S$, the optimal expected reward is the \emph{(negative) entropy} induced by $S$,
\[
  \Phi:\mathcal P\To\R,\quad \Phi(p)=p\cdot S(p),
\]
\citep{PDL}. In what follows, we refer to $\Phi$ simply as the entropy function associated to $S$, as there is no danger of confusion between negative and positive entropy functions in the present context. The \emph{regret} for quoting $q$ instead of the true density $p$ is expressed by the function
\begin{equation*}
  D:\mathcal P\times\mathcal P\To\R,\quad D(p,q) = p\cdot S(p)- p\cdot S(q),
\end{equation*}
which in the statistics literature is also known as the \emph{divergence} induced by $S$. In the present paper, we shall use the notions of entropy and divergence in a more general sense by replacing strict propriety with propriety.

General overviews of proper scoring rules may be found in \cite{GR, GK} in connection to probabilistic forecasting, and also in \cite{DawMus1}, where the emphasis is on statistical inference. Theoretical aspects of proper scoring rules are studied in \cite{D, GD, Will}. \cite{FK} investigate proper scoring rules in connection with the elicitation of private information. The remaining references throughout the text provide links to more specific uses of scoring rules.

\subsection{Motivation and Scope of the Paper}

In this paper we adopt the theoretical framework of \citet{HB}. This approach is characterised by exploiting a beautiful connection with \emph{Euler's homogeneous function theorem}, which presupposes that we extend our quantities of interest as homogeneous functions to the conic hull of the prediction set. To that end, we introduce the \emph{prediction cone} $\mathcal P^+=\{\lambda p \,|\, \lambda>0,\, p\in\mathcal P\}$ and extend $S$ and $\Phi$ to $\mathcal P^+$ as homogeneous functions of degrees zero and one, respectively. Any $\mathcal P$-integrable function $q^*$ satisfying
\begin{equation*}
  \Phi(p) \geq p\cdot q^*, \quad \forall p\in\mathcal P^+,
\end{equation*}
with equality for $p=q$, is called a \emph{$\mathcal P$-integrable subgradient} of $\Phi$ at $q$. The subgradient is called \emph{strict} if the above inequality is strict for all $p\in\mathcal P^+$ not positively collinear to $q$. Suppose that $\Phi$ has a subgradient $S(q)\in\mathcal L(\mathcal P)$ at each $q\in\mathcal P^+$ and the resulting map $S:\mathcal P^+\To \mathcal L(\mathcal P)$ is homogeneous of degree zero. We call $S$ a $\mathcal P$-\emph{integrable subgradient} of $\Phi$ on $\mathcal P^+$. We recall that a (strictly) convex homogeneous function of degree one is a \emph{(strictly) sublinear function}. We may now state Hendrickson and Buehler's classical result in a slightly more contemporary language.

\begin{thm}\label{thm: HB}
Let $\mathcal P$ be a prediction set with respect to the measure space $(\Omega,\mathcal A,\mu)$. A scoring rule $S: \mathcal P^+\To \mathcal L(\mathcal P)$ is (strictly) proper if and only if there is a (strictly) sublinear function $\Phi:\mathcal P^+\To\R$ such that $S$ is a subgradient of $\Phi$ on $\mathcal P^+$.
\end{thm}

Theorem \ref{thm: HB} provides us with a basic but insufficient theoretical framework to discuss the titular question of this paper. In support of this claim, in Example \ref{exmpl: sup} we show the existence of a sublinear function that has unique but non-$\mathcal P$-integrable subgradients at some points of its domain, while at other points it has multiple $\mathcal P$-integrable subgradients. The most important structure missing in Theorem \ref{thm: HB} is the notion of \emph{interior} of a convex domain, which lies at the intersection of geometry, algebra, and topology, and may have different incarnations depending on the context  \citep{BVff, Rock}. For example, studying proper local scoring rules on discrete sample spaces, \citet{DLP} apply Theorem \ref{thm: HB} in a context where the prediction cone is the interior of the positive orthant in $\R^d$. In this case, well-known results from convex analysis give necessary and sufficient conditions for an affirmative answer to our basic question. The real focus of our paper is thus the non-Euclidean case in the abstract measure-theoretic setting introduced above.

In Proposition \ref{prop: subdiff} and Example \ref{exmpl: Hahn-B}, we show that at boundary points sublinear functions have either no subgradient, or infinitely many. Therefore, it is paramount to try to define entropy functions on interiors of positive cones. In infinite dimensions, however, this is not always possible. Indeed, it is well-known that the positive cones in many natural function spaces (such as the Lebesgue $L^p$-spaces over $\R^d$) have empty interiors \citep{BorLew} and are negligible sets in terms of Baire category. This calls for a more subtle approach to our problem in which we need to refine our notion of interior and boundary. Inspired by geometric functional analysis, we adapt an algebraic refinement of the notion of interior of convex sets, whose better known topological analogues are often referred to as \emph{quasi-interior} \citep{FullBraun, BorLew}. Common entropies whose domains are positive cones with empty interior but nonempty quasi-interior are the \emph{Shannon entropy}, the \emph{\Hy\ entropy}, and in principle, the entropies associated with the \emph{proper local scoring rules of arbitrary orders}. These entropies are formally not differentiable functions but possess directional derivatives on large subspaces, which display similar properties to standard gradients.

Other entropies, such as those that are associated with the families of \emph{power scoring rules} and \emph{pseudospherical scoring rules} may be extended continuously to open cones in normed spaces that contain signed densities. Geometrically, this setting is similar to the Euclidean setting. One applies the supporting hyperplane theorem and other standard results in analysis relating subgradients and G\^ateaux derivatives. The latter entropies are G\^ateaux differentiable (either everywhere or outside a negligible set), which we illustrate in the context of the \emph{quadratic scoring rule}.

The original part of the paper is concerned with the analysis of the notion of $\mathcal P$-integrable subgradient introduced by \cite{HB} and the associated most basic general framework for proper scoring rules. To address the question of existence and uniqueness of proper scoring rules, we equip this framework with a notion of algebraic quasi-interior. As an illustration, we show that the Hyv\"arinen scoring rule is the unique 0-homogeneous $\mathcal P$-integrable subgradient of its entropy function on the (non-empty) quasi-interior of a suitable positive cone.

The paper is organised as follows. In Section \ref{sect: not}, we introduce the notation and present all the  background facts. Section \ref{sect: main} contains our main results which formulate necessary and sufficient conditions for existence and uniqueness of subgradients of entropy functions. In Section \ref{sect: exmpl}, we illustrate the theory with applications to three prototypical entropy functions, namely, the Shannon, \Hy, and quadratic entropy. These examples formalise the meaning with respect to which we may consider each entropy to have a uniquely associated proper scoring rule. We complete the main part of the paper in Section \ref{sect: conc} with some closing remarks. The proofs of all formal assertions made in the text are given in Appendix \ref{app: proofs}. In Appendix \ref{app: exmpl}, we present additional facts that illustrate various points made in the Introduction or later in the text.

\section{Notation and Preliminaries}\label{sect: not}

Let $E$, $E_1$, $E_2$ be sets of $\mu$-measurable functions on $\Omega$. For $\alpha\in\R$, we use the notation
\begin{align*}
  \alpha E_1 &= \{\alpha {f}\,|\, {f}\in E_1\} \\
  E_1 + E_2 &= \{{f}+{g}\,|\, {f}\in E_1, {g}\in E_2\}.
\end{align*}
The (\emph{blunt}) \emph{cone} of $E$ is the set $E^+=\{\lambda {f}\,|\, \lambda>0, {f}\in E\}$, while the \emph{pointed cone} of $E$ is the set $E^+\cup\{0\}$. The \emph{convex hull} of $E$,
\[
\Co E = \left\{\sum _{{i=1}}^{k}\alpha _{i}{f}_{i}{\Bigg |}k\geq 1,\,{f}_{i}\in E,\,\alpha _{i}\geq0,\,\sum _{{i=1}}^{k}\alpha _{i}=1\right\},
\]
is the set of all convex combinations of elements of $E$. The \emph{conic hull} of $E$,
\[
\Cone E = \left\{\sum _{{i=1}}^{k}\alpha _{i}{f}_{i}{\Bigg |}k \geq 1,\,{f}_{i}\in E,\,\alpha _{i}\geq0\right\},
\]
is the set of all conic combinations of elements of $E$. By
\[
\Span E = \left\{\sum _{{i=1}}^{k}\alpha _{i}{f}_{i}{\Bigg |}k \geq 1,\,{f}_{i}\in E,\,\alpha _{i}\in\R\right\}
\]
we denote the set of all linear combinations of elements of $E$, and we refer to it as the \emph{linear span} of $E$.

A set $E$ is called \emph{convex} if $\Co E=E$, a \emph{cone} if $E=E^+$ or $E=E^+\cup\{0\}$, a \emph{convex cone} if $E=\Cone E$ or $E=\Cone E\setminus\{0\}$, and  a \emph{linear space} if $E=\Span E$. If $E$ is convex, $ E^+ = \Cone  E\setminus\{0\}$ is a convex cone.

The \emph{epigraph} of $\Phi: E \To\R$ is the set in $\Span E \times \R$ given by
\[
  \Epi \Phi = \{ ({f},y) \,|\, {f}\in E , \, y\in\R,\, y\geq\Phi({f}) \}.
\]
The \emph{graph} of $\Phi$ is the set $\{({f},\Phi({f})) \,|\, {f}\in E\}$.

A function $\Phi: E\To\R$ is called \emph{convex} if its epigraph is a convex set. The definition implies that $E$ is convex. Therefore, $\Phi$ is convex if, for any ${f},{g}\in{E}$ and $\lambda\in(0,1)$, $\Phi$ satisfies
\begin{equation*}
  \Phi((1-\lambda){f}+\lambda {g}) \leq (1-\lambda) \Phi({f}) + \lambda\Phi({g}).
\end{equation*}
If the inequality is strict for ${f}\not={g}$, then $\Phi$ is called \emph{strictly convex}.

A function $\Phi:{E^+}\To\R$ is said to be \emph{(positively) homogeneous of degree $k$}, for $k\in\R$, or \emph{(positively)} $k$-\emph{homogeneous}, if for every $f\in E^+$ and every $\lambda>0$, it holds $\Phi(\lambda f) = \lambda^k \Phi(f)$. A function $\Phi: E\To\R$ is said to be \emph{subadditive} if $\Phi$ satisfies
\[
 \Phi( {f}+ {g}) \leq \Phi( {f}) + \Phi({g})
\]
for all ${f}, {g}\in E$, and \emph{strictly subadditive}, if the above inequality is strict for $f\not=g$. We need to modify slightly the latter definition in the case when $\Phi: E^+\To\R$ is 1-homogeneous. Then we say that $\Phi$ is \emph{strictly subadditive} if the above inequality is strict whenever $f,g\in E^+$ are not positively collinear. Functions that are 1-homogeneous and (strictly) subadditive are called \emph{(strictly) sublinear}. It is easy to see that $\Phi:E^+\To\R$ is (strictly) sublinear if and only if $\Phi$ is (strictly) convex on $E$ and 1-homogeneous on $E^+$.

Let $\mathcal P$ be a prediction set with respect to $(\Omega,\mathcal A, \mu)$ and let $E\subset\Span \mathcal P$. By $E^\perp$ we denote the \emph{annihilator} of $E$ in $\mathcal L(\mathcal P)$, that is, all $f\in\mathcal L(\mathcal P)$ such that
\begin{equation*}
  p \cdot f = 0
\end{equation*}
for all $p\in E$. Clearly, $E^\perp$ is a linear subspace of $\mathcal L(\mathcal P)$. In the case when $E^\perp=\{0\}$, we say that $E$ has a \emph{trivial annihilator}.

By a \emph{direction} in a vector space we understand the equivalence class of all positively collinear vectors to a given nonzero vector. Note that any 0-homogeneous function is a function of directions. For ${q}\in {\mathcal P^+}$, we define the \emph{set of directions} from ${q}$ to the points in ${\mathcal P^+}$ as
\begin{align*}
  \mathcal D({q}) &= \{{p} \in \Span {\mathcal P}\,|\, \exists \epsilon_{{p}}>0,\, \forall t\in[0,\epsilon_{{p}}],\, {q}+t{} {p}\in \mathcal P^+\}\\
       &= \{ {p}\in \Span {\mathcal P} \,|\, \exists \epsilon_{{p}}>0,\, {q}+\epsilon_{p} {p}\in \mathcal P^+\}.
\end{align*}
We have the latter identity due to the convexity of ${\mathcal P^+}$.

A point ${q}\in \mathcal P^+$ is an \emph{algebraically interior} point of $\mathcal P^+$ if ${\mathcal D}({q})=\Span {\mathcal P}$. The collection of all algebraically interior points of $\mathcal P^+$ is called the \emph{algebraic interior} of $\mathcal P^+$. In the case of a topological vector space, the topological interior of a set is always contained in the algebraic interior of the set. Moreover, when the topological interior is not empty, the two notions coincide. If $q$ is not algebraically interior for $\mathcal P^+$, that is,  ${\mathcal D}({q})\not=\Span {\mathcal P}$, we say that $q$ is a \emph{boundary point} for $\mathcal P^+$. If $\mathcal P^+$ has empty algebraic interior, then the prediction cone consists entirely of boundary points. This case occurs frequently in the context of continuous sample spaces, see e.g. Proposition \ref{prop: L1}.

\begin{lem}\label{lem: Dir}
For each ${q}\in \mathcal P^+$, we have the representation
\[
  {\mathcal D}( {q}) = \Cone(\mathcal P^+ -  {q}).
\]
\end{lem}

For a point $q\in\mathcal P^+$, we define $\mathcal O(q) = {\mathcal D}( {q}) \cap -{\mathcal D}( {q})$. This is the subset of directions in $\mathcal D(q)$ whose inverse is also in $\mathcal D(q)$. The set may be identified with these directions in $\Span \mathcal P$ along which there is an open line segment that contains $q$ and is contained in $\mathcal P^+$. Clearly, $q$ is algebraically interior for $\mathcal P^+$ if and only if $\mathcal O(q)=\mathcal D(q)=\Span\mathcal P$. By construction, $\mathcal O(q)$ is a linear subspace of $\Span \mathcal P$. The sets of directions $\mathcal D(q)$ and $\mathcal O(q)$ are instrumental for defining various notions of directional derivatives.

The most basic directional derivative is the following one.

\begin{defn}\label{defn: one-sided grad}
For a function $\Phi: \mathcal P^+\To\R$, the \emph{right directional derivative} of $\Phi$ at $q\in\mathcal P^+$ along $p\in\mathcal D(q)$ is defined as
\begin{equation}\label{eq: right}
  \Phi_+'(p, q) = \lim_{t\To 0^+} \frac{\Phi(q+t p)-\Phi(q)}{t}
\end{equation}
if the limit exists.
\end{defn}

We gather below the main properties of $\Phi_+'(p, q)$.

\begin{prop}\label{prop: Phi+}
Let $\Phi:\mathcal P^+\To\R$ be a sublinear function and $q\in\mathcal P^+$. We have
\begin{enumerate}[(a)]
\item  for each $p\in\mathcal D(q)$,
\begin{equation*}
  \Phi_+'(p, q) = \inf_{t>0} \frac{\Phi(q+t p)-\Phi( q)}{t}\in\R\cup\{-\infty\},
\end{equation*}
and the infimum is finite for $p\in\mathcal O(q)$;
\item
$\Phi_+'(\cdot,q):\mathcal D(q)\To\R\cup\{-\infty\}$ is sublinear;
\item for each $\lambda>0$, $\Phi_+'({} p,\lambda {} q)=\Phi_+'({} p,{} q)$;
\item for each $p\in\mathcal P^+$, 
\begin{align*}
  \Phi( p) \geq {\Phi_+'}( p,  q),
\end{align*}
with equality for $p=q$;
\item for each $p\in\mathcal O(q)$, $-\Phi_+'(-p,q)\leq\Phi_+'(p,q)$;
\item the set
\begin{equation*}
  \mathcal O'(q) = \{p\in\mathcal O(q)\,|\, -\Phi_+'(-p,q)=\Phi_+'(p,q)\}
\end{equation*}
is a linear subspace of $\mathcal O(q)$ and the restriction $\Phi_+'(\cdot,q)\big|_{\mathcal O'(q)}$ is linear.
\end{enumerate}
\end{prop}

We next consider the other two types of directional derivatives. First, if we take the limit \eqref{eq: right} with the restriction $t\leq0$ instead $t\geq 0$, we obtain the \emph{left directional derivative} of $\Phi$, denoted $\Phi_-'(\cdot,q)$. It is easy to see that $\Phi_-'(\cdot,q)$ can be defined on $\mathcal O(q)$ and we have $\Phi_-'(p,q) = -\Phi_+'(-p,q)$, for each $p\in\mathcal O(q)$. Thus part (e) above can be rewritten as
\[
    \Phi_-'(p,q) \leq \Phi_+'(p,q)
\]
for all $p\in\mathcal O(q)$. On the subspace $\mathcal O'(q)$ introduced above in part (f), we have that
\[
  \Phi_-'(\cdot, q) = \Phi_+'(\cdot, q)
\]
is in fact the \emph{two-sided directional derivative} of $\Phi$ at $q$, denoted $\Phi'(\cdot,q)$. The latter can be defined as the limit \eqref{eq: right} without any restriction on $t$. In the most important case in practice, we have that $\mathcal O(q)=\mathcal O'(q)$. If in addition $\mathcal O(q)\not=\Span\mathcal P$, then $\Phi$ has no standard functional derivative. For an illustration of this fact in the context of Shannon and {\Hy } entropies, see  Section \ref{sect: exmpl}.

By $\Lin \mathcal P$ we denote the space of all real-valued linear functionals on $\Span\mathcal P$, i.e., the algebraic dual of $\Span\mathcal P$. By ``$\cdot$" we denote the bilinear pairing on $\Span\mathcal P \times \Lin \mathcal P$, so if $q\in\Span\mathcal P$ and $q^*\in\Lin \mathcal P$, $q\cdot q^*$ is the value of $q^*$ at $q$.

Let $\Phi:\mathcal P^+\To\R$ be 1-homogeneous. We say that $q^*\in\Lin\mathcal P$ is a \emph{subgradient} of $\Phi$ at $q$ if
\begin{equation*}
  \Phi(p) \geq  p \cdot q^*
\end{equation*}
for all $p\in\mathcal P^+$, with equality for $p=q$. The collection of all subgradients of $\Phi$ at $q$ is called the \emph{subdifferential} of $\Phi$ at $q$ and is denoted by $\partial\Phi(q)$. A subgradient $q^*$ is \emph{strict} if and only if the inequality $\Phi(p) > p\cdot q^*$ holds for all $p\in\mathcal P^+$ not positively collinear with $q$.

If $h\in\Lin\mathcal P$, the \emph{hyperplane} $H$ in $\Span\mathcal P\times\R$ given by
\begin{equation*}
  z=p\cdot h,\quad \forall p\in\Span\mathcal P,
\end{equation*}
\emph{supports} $\Phi$ at $q$ if the epigraph of $\Phi$ lies above $H$, and $H$ contains the point $(q,\Phi(q))$. Clearly, $H$ supports $\Phi$ at $q$ if and only if $h\in\partial\Phi(q)$.

The following proposition describes the intimate connection between one-sided and two-sided directional derivatives and the subdifferential of a sublinear function.

\begin{prop}\label{prop: subdiff}
For a point $q\in\mathcal P^+$,  we have
\begin{enumerate}[(a)]
\item
$q^*\in \partial\Phi(q)$ if and only if
\begin{equation*}
  p \cdot q^* \leq \Phi_+'(p,q)
\end{equation*}
for all $p\in\mathcal P^+$, with equality for $p=q$;
\item if $\mathcal D(q)=\Span\mathcal P$ and $\Phi'(\cdot,q)$ exists on $\Span\mathcal P$, then $\partial\Phi(q)=\{\Phi'(\cdot,q)\}$;
\item if $\mathcal D(q)=\Span\mathcal P$ and $\Phi'(\cdot,q)$ does not exist on $\Span\mathcal P$, then $\partial\Phi(q)$ has multiple elements;
\item if $\mathcal D(q)\not=\Span\mathcal P$ and $\Phi_+'(p,q)$ is finite for all $p\in\mathcal P^+$, then $\partial\Phi(q)$ has multiple elements;
\item if $\mathcal D(q)\not=\Span\mathcal P$ and there is $p\in\mathcal P^+$ such that $\Phi_+'(p,q)=-\infty$, then $\partial\Phi(q)=\emptyset$.
\end{enumerate}
\end{prop}

Part (a) above is the standard characterisation of the subdifferential of a sublinear function. Parts (b) and (c) give additional information in the case of algebraically interior points. Parts (d) and (e) do the same for boundary points. Notice that the latter imply the statement from the Introduction that at boundary points either the existence or uniqueness of subgradient fails. (See also Example \ref{exmpl: Hahn-B}.) In the next section, we show that uniqueness might be sometimes recovered at certain boundary points if we confine ourselves to a regularity class such as $\mathcal L(\mathcal P)$.

We next give a formal definition of a scoring rule and elaborate some of its implications.

\begin{defn}\label{defn: score}
Let $\mathcal P$ be a prediction set with respect to the measure space $(\Omega,\mathcal A,\mu)$. Any 0-homogeneous map $S:\mathcal P^+\To\mathcal L(\mathcal P)$ is called a \emph{scoring rule}.
\end{defn}

If $X$ is a random variable on $\Omega$ with unknown true density $p\in\mathcal P$, then for each predictive density $q\in\mathcal P^+$, $S(q)(X)$ is a random function of $X$. The condition $S(q)\in\mathcal L(\mathcal P)$ guarantees that the expectation of $S$ is always finite. The \emph{uncertainty function} associated to $S$ is the function $\Phi:\mathcal P^+\To\R$, $\Phi(p)=p\cdot S(p)$. Clearly, $\Phi$ is 1-homogeneous. When $S$ is proper, it is customary to call $\Phi$ an \emph{entropy function}. 

Suppose now that $S:\mathcal P^+\To \mathcal L(\mathcal P)$ is a proper scoring rule with entropy $\Phi$. The condition that the expected score of $S$ is maximised in $q$ at the true density $q=p$ means that $S$ satisfies the inequality
\begin{equation*}
  \Phi(p) \geq p\cdot S(q),
\end{equation*}
for each $p,q\in\mathcal P^+$, with equality for $q=p$.  If $S$ is strictly proper, then $p$ is the only maximiser up to a scaling factor. In this case, the inequality above is strict for any $q$ that is not positively collinear to $p$. So, the assumption of propriety is equivalent to $S$ being a subgradient of $\Phi$ on $\mathcal P^+$. Moreover, strict propriety corresponds to strict subgradients on $\mathcal P^+$. The existence of a subgradient on $\mathcal P^+$ implies that $\Phi$ is sublinear, see Lemma \ref{lem: subg-subl}. We conclude that (strictly) proper scoring rules are $\mathcal P$-integrable subgradients of (strictly) sublinear functions. Therefore, it is reasonable in the context of scoring rules to restrict the notion of subgradient to the class $\mathcal L(\mathcal P)\subset \Lin(\mathcal P)$. In the next section, and in particular in Theorem \ref{thm: exist} and Theorem \ref{thm: uniq}, we discuss the existence and uniqueness of $\mathcal P$-integrable subgradients.

In some special cases, we may add to our notion  of subgradient a topological structure. Let $\mathcal P^+$ be a prediction cone such that $\Span \mathcal P$ may be identified with a normed space $(N,\norm\cdot)$, and let the continuous dual of $N$, denoted $N^*$, be a subset of $\mathcal L(\mathcal P)$. Suppose that $\mathcal P^+\subset\mathcal C$, where $\mathcal C$ is an open convex cone in $N$, and $\Phi$ may be extended to $\mathcal C$ as a continuous sublinear function. 

We recall that $\Phi$ is  \emph{G\^ateaux differentiable}  at $q\in \mathcal C$ if there is $q^*\in N^*$ such that for every $p\in N$, the limit
\begin{equation*}
  p \cdot q^* = \lim_{t\To0} \frac{\Phi(q+tp)-\Phi(q)}{t}
\end{equation*}
exists. The functional $q^*$ is called the \emph{G\^ateaux derivative} of $\Phi$ at $q$ and is also denoted by $\nabla\Phi(q)$. Notice that by definition the G\^ateaux derivative is applicable only to interior points. See Theorems \ref{thm: N1} and \ref{thm: N2} for an answer to our two main questions.

If $\Phi$ is G\^ateaux differentiable at $q$, taking $p=q$ in the above limit, we recover \emph{Euler's homogeneous function} theorem
\begin{equation*}
  q\cdot \nabla \Phi(q) = \Phi(q).
\end{equation*}
More generally, if $\Phi$ is sublinear and has a subgradient $S$ on $\mathcal P^+$, then we have that $q\cdot S(q) = \Phi(q)$, for every $q\in\mathcal P^+$, \citep{HB}. The proof also follows from Proposition \ref{prop: subdiff} (a) and Proposition \ref{prop: Phi+} (d). This beautiful generalisation of Euler's theorem is only visible after extending $S$ and $\Phi$ to denormalised densities as homogeneous functions.

Suppose now that a scoring rule $S:\mathcal P\To\mathcal L(\mathcal P)$ is given. Then, setting
\begin{equation*}
  S(q) =  S\left(\frac q {q\cdot 1}\right)
\end{equation*}
for any $q\in\mathcal P^+$, extends $S$ as a 0-homogeneous function to the prediction cone. Here
\[
   q \cdot 1 = \int_{\Omega} q(x)d\mu(x)
\]
is the normalising constant of $q$. Similarly, let an entropy function $\Phi:\mathcal P\To\R$ be given. Setting
\begin{equation*}
  \Phi(q) = (q\cdot 1) \Phi\left(\frac q {q\cdot 1}\right)
\end{equation*}
for any $q\in\mathcal P^+$, extends $\Phi$ as a 1-homogeneous function to the prediction cone. See Section \ref{sect: exmpl} for an illustration. Working directly with denormalised predictive densities could also be advantageous in numerical computation \citep{Hyv, Hyv1, DawMus, DawMus1}.

\section{Main Results}\label{sect: main}

Our first result gives a necessary and sufficient condition for existence of a $\mathcal P$-integrable subgradient at a point. The result can be easily  generalised to subgradients on $\mathcal P^+$.

\begin{thm}\label{thm: exist}
Let $\Phi:\mathcal P^+\To\R$ be a sublinear function. Then $\Phi$ has a $\mathcal P$-integrable subgradient at a point $q\in\mathcal P^+$ if and only if there is $q^*\in\mathcal L(\mathcal P)$ such that
\begin{equation*}
  p\cdot q^* \leq \Phi_+'(p,q)
\end{equation*}
for all $p\in\mathcal P^+$, with equality for $p=q$.
\end{thm}

In the light of Theorem \ref{thm: HB} and the above result, we call any sublinear function $\Phi$ an \emph{entropy} if $\Phi$ has a $\mathcal P$-integrable subgradient at each point of its domain. In most cases of practical interest, one may choose the prediction cone appropriately so that $\Phi'_+(\cdot,q)=q^*$ for some $q^*\in\mathcal L(\mathcal P)$. This means that $\Phi'_+(\cdot,q)$ is a $\mathcal P$-integrable subgradient of $\Phi$ at $q$ and that $\Phi'_+(\cdot,q)=\Phi'(\cdot,q)$ is also a two-sided directional derivative on the subspace $\mathcal O(q)$ of $\Span\mathcal P$. In our next result, we show that if $\mathcal O(q)$ is a sufficiently large subspace, then $\Phi'_+(\cdot,q)$ is the unique $\mathcal P$-integrable subgradient of $\Phi$ at $q$.

\begin{thm}\label{thm: uniq}
Let $\mathcal P$ be a prediction set and $\Phi:\mathcal P^+\To\R$ be a sublinear function. Suppose that at a point $q\in\mathcal P^+$  the subspace $\mathcal O(q)$ of $\Span\mathcal P$ has a trivial annihilator in $\mathcal L(\mathcal P)$. If there is a $q^*\in\mathcal L(\mathcal P)$ such that
\begin{equation}\label{eq: Phi+ q*}
  p\cdot q^* = \Phi_+'(p,q)
\end{equation}
for all $p\in\mathcal P^+$, then $q^*$ is the unique $\mathcal P$-integrable subgradient of $\Phi$ at $q$.
\end{thm}

In the above result, the condition that $\mathcal O(q)$ has a trivial annihilator in $\mathcal L(\mathcal P)$ can be interpreted to say that the set of directions at which $q\in\mathcal P^+$ is boundary to the cone $\mathcal P^+$ is negligible. The latter condition represents an algebraic analogue to the property of $q$ being a \emph{quasi-interior point} of $\mathcal P^+$, which is better known in its topological forms presented in \cite{FullBraun, BorLew}. The collection of all quasi-interior points of $\mathcal P^+$ is the \emph{quasi-interior} of $\mathcal P^+$. As an illustration, in the next section we define Shannon and Hyv\"arinen entropies on positive cones with nonempty quasi-interiors. Presently, however, we do not investigate the proposed variant of quasi-interior in full. This analysis is not necessary for the application of Theorem \ref{thm: uniq} and may be a subject of future work. Notice also that uniqueness of subgradient is understood and valid only within the class $\mathcal L(\mathcal P)$.

We now consider the case of topological subgradients. Our main assumption is the following:
\begin{align}\label{eq: cont ass}
\begin{cases}
    \mathcal P^+\subset \mathcal C, \text{ where }\mathcal C \text{ is an open convex cone in a normed space } N\\
   \Phi: \mathcal C\To\R \text{ is a continuous sublinear function}.
\end{cases}
\end{align}

\begin{thm}\label{thm: N1}
If \eqref{eq: cont ass} holds, then $\Phi$ admits a subgradient $S: \mathcal C \To N^*$.
\end{thm}

The result is generally known as the \emph{supporting hyperplane theorem}. For proof see e.g. \cite{NicPer, BVff, Zal, Rud}. Any subgradient $S: \mathcal C \To N^*$ of $\Phi$ may be identified with a proper scoring rule on $\mathcal P^+$ by restricting $S$ to $\mathcal P^+$.

\begin{thm}\label{thm: N2}
Assume \eqref{eq: cont ass}. Then, $\Phi$ is G\^ateaux differentiable on $\mathcal C$ if and only if $\Phi$ admits a unique  subgradient $S: \mathcal C\To N^*$. In this case $S=\nabla\Phi$ is the G\^ateaux derivative of $\Phi$.
\end{thm}

This is a standard result in convex analysis. See e.g. \cite{BVff, Zal}. See Example \ref{exmpl: sup} for an illustration of the case where the assumption $N^*\subset \mathcal L(\mathcal P)$ is not satisfied.

\section{Applications}\label{sect: exmpl}

In this section, we apply our main results to three important entropies: \emph{Shannon entropy}, \emph{{\Hy } entropy}, and \emph{quadratic entropy}. For each entropy, we investigate an appropriate domain with nonempty quasi-interior for which we show the existence of a unique subgradient. 

\subsection{Shannon Entropy}\label{sect:  Sha}

The \emph{Shannon entropy function} for densities on $\R^d$ is given by
\begin{equation}\label{eq: Sha}
  \Phi(p) = \int_{\R^d} p(x) \ln \frac {p(x)} {p\cdot 1} dx
\end{equation}
where $p(x)\geq0$ is assumed to be sufficiently regular. More facts about Shannon entropy may be found e.g. in \cite{B, D, PDL, DLP}.

We first show that Shannon entropy may only be defined for nonnegative functions in a natural way. The kernel of $\Phi$ is the function $\phi(t)=t\ln t$ for $t>0$ and $\phi(0)=0$. Clearly, $\phi(t)$ is strictly convex on $t\geq0$ since, for $t>0$, $\phi''(t)=1/t>0$, and $\phi$ is continuous at the endpoint $t=0$. Notice that $\phi(t)$ has a vertical tangent at $t=0$ since $\phi'(t)=\ln t+1$. We conclude that $\phi(t)$ cannot be extended as a convex function to $t<0$. This furnishes our claim.

The positive cone of $L^1(\R^d)$ comprises of all nonnegative functions in $L^1(\R^d)$ and is denoted by $L^1_+(\R^d)$. In Proposition \ref{prop: L1} we give a direct proof that $L^1_+(\R^d)$ is a nowhere dense subset of $L^1(\R^d)$. Since the domain of Shannon entropy is a subset of $L^1_+(\R^d)$, it too is a nowhere dense set.

We now proceed to find a suitable prediction set. For $a\geq d+1$, we set
\begin{align*}
  \mathcal P^+ = \left\{p \in C(\R)\,\Big|\, p(x)>0\,, \exists C_1,C_2>0\, : \right. \left. \, \frac {C_1} {(1+\abs x)^a} \leq {p(x)}\leq \frac {C_2} {(1+\abs x)^{d+1}}\right\}.
\end{align*}
Notice that $\mathcal L(\mathcal P) \subset L^1_{\Loc}(\R^d)$. Indeed, for any $f\in\mathcal L(\mathcal P)$ consider
\begin{equation*}
  p_t(x) = \begin{cases}
            1                       \qquad &0<\abs x<t\\
            \left(\frac{1+t}{1+\abs x}\right)^{d+1}     \qquad &t\leq \abs x.
         \end{cases}
\end{equation*}
Since $p_t\in\mathcal P^+$, the $\mathcal P$-integrability of $f$ implies that
\[
   \int_{\abs x\leq t} \abs{f(x)}dx < \infty
\]
for all $t>0$.

Let us next see that for any $q\in\mathcal P^+$, $\mathcal O(q)$ has a trivial annihilator in $\mathcal L(\mathcal P)$. Clearly, $\mathcal O(q)$ contains all $p\in\Span\mathcal P$ that have faster or equal decay at infinity compared to $q$. Suppose that $f\in \mathcal O(q)^\perp$. Choosing an appropriate approximation of the identity, $\{p_n\}$, $p_n\in\mathcal O(q)$, we get that $f*p_n(x)\To f(x)$ for every $x$ in the Lebesgue set of $f$. Hence $f=0$ a.e. on $\R^d$. We conclude that $\mathcal O(q)^\perp=\{0\}$.

After this preparation, we may now define $\Phi$ rigorously as the map from $\mathcal P^+$ to $\R$ given by \eqref{eq: Sha}. Strict convexity of $\Phi$ follows from the strict convexity of $t\ln t$, for $t\geq0$, while its 1-homogeneity is trivial. Therefore, $\Phi$ is strictly sublinear on $\mathcal P^+$. Let us compute the right directional derivative of $\Phi$.

For $q\in\mathcal P^+$ and $p\in\mathcal D(q)$, we set $q_t=q+tp$. We have
\begin{align*}
  \lim_{t\To0^+} \frac{\Phi(q+tp) - \Phi(q)}{t} & = \frac d {dt}\Bigg|_{t=0} \left(q_t \cdot \ln \frac{q_t}{q_t\cdot1}\right)\\
                                            & = p \cdot \ln \frac{q}{q\cdot 1} + q \cdot \left(\frac p q - \frac { p\cdot 1} {q\cdot 1}\right)\\
                                            &= p \cdot \ln \frac{q}{q\cdot 1}.
\end{align*}
Therefore,
\begin{equation*}
  \Phi_+'(p,q) = \int_{\R^d} p(x) \ln \frac {q(x)} {q\cdot 1} dx.
\end{equation*}
Clearly, the function
\[
   S(q)(x) = \ln \frac {q(x)} {q\cdot 1}
\]
is in $\mathcal L(\mathcal P)$. Indeed, the claim follows from the fact that $S(q)$ is continuous in $x$ and grows logarithmically as $\abs x\To\infty$. In view of Theorem \ref{thm: uniq}, $S$ is the unique $\mathcal P$-integrable subgradient of $\Phi$ on $\mathcal P^+$ since $\Phi_+'(p,q)=p\cdot S(q)$ for every $p,q\in\mathcal P^+$. The map is known as the \emph{logarithmic scoring rule}.

The uniqueness of the logarithmic scoring rule as a subgradient of Shannon entropy is in no way an absolute fact. Using the Hahn-Banach theorem as illustrated in Example \ref{exmpl: Hahn-B} and the fact that $L^1_+(\R^d)$ consists entirely of boundary points, one may construct other subgradients of $\Phi$ that lie outside $\mathcal L(\mathcal P)$. Moreover, if $q$ lies on the \emph{quasi-boundary} of $\mathcal P^+$ (i.e. the points where the condition $\mathcal O(q)^\perp=\{0\}$ is violated), then uniqueness will fail even within $\mathcal L(\mathcal P)$.

\subsection{{\Hy } Entropy}\label{sect: Fi}

\emph{{\Hy }  entropy} for densities on $\R^d$ is defined as
\begin{equation}\label{eq: Fi}
  \Phi(p)=\int_{\R^d} \frac{\abs{\nabla p(x)}^2}{p(x)}dx.
\end{equation}
Here $\nabla$ is the gradient on $\R^d$. {\Hy } and related entropies are considered e.g. in \cite{PDL,EG, ForbLau, DawMus, Hyv, Hyv1, Sanch}.

We first show that there is no natural way to extend {\Hy } entropy to signed densities. For simplicity, we confine ourselves to the case $d=1$. Suppose that $p$ changes sign at some $x_0\in\R$ that has multiplicity one. The assumption is generic and it means that $x_0$ is not an inflection point of $p$. It follows that the above integral is divergent at $x_0$. Indeed, the claim is a direct consequence of the asymptotic expansion of the term
\[
   \frac{\abs {p'(x)}^2}{p(x)} = \frac 1 {x-x_0} + O(x-x_0)
\]
near $x_0$. On the other hand, if $p$ has a zero of higher multiplicity at $x_0$, one may check that the above asymptotics will be bounded and the integral will be convergent in a neighbourhood of $x_0$. Nevertheless, the example shows that $\Phi$ cannot be generally defined for densities that change sign.

We proceed to define a suitable domain for $\Phi$. Suppose that $\mathcal P^+$ consists of all positive, twice continuously differentiable functions $p(x)$ on $\R^d$ that satisfy the bounds:
\begin{enumerate}[(a)]
  \item there are $C_1>0$ and $k>0$ such that
  \begin{equation*}
    \abs{\frac{\nabla p(x)}{p(x)}} + \abs{\frac{\Delta p(x)}{p(x)}} \leq C_1 (1+\abs x)^k;
  \end{equation*}
  \item there is $C_2>0$ such that
  \begin{equation*}
    \abs{p(x)} \leq  \frac {C_2} {(1+\abs x)^{d+1+k^2}},
  \end{equation*}
\end{enumerate}
where $\Delta=\partial^2/\partial x_1^2+\cdots+\partial^2/\partial x_d^2$ is the \emph{Laplacian} on $\R^d$. In view of the above, we have the following limit
\begin{equation}\label{eq: lim 1}
  \lim_{R\To\infty} \frac 1 R \int_{\abs{y}=R}  \left(\frac{y \nabla q(y)}{q(y)} \right)p(y)dy = 0
\end{equation}
for any $p,q\in\mathcal P^+$. Note that here
\[
  y\nabla q(y)= y_1\frac{\partial q(y)}{\partial y_1}+\cdots+ \frac{y_d\partial q(y)}{\partial y_d}
\]
denotes the \emph{scalar product} of $y$ and $\nabla q(x)$ and the integral in \eqref{eq: lim 1} is a surface integral over the sphere centred at the origin of radius $R$. The class $\mathcal P$ is broad, e.g. it contains the Gaussians, and all positive continuous densities that have bounded first and second-order derivatives and decay at infinity sufficiently fast. Just like in Section \ref{sect:  Sha}, we have that $\mathcal L(\mathcal P)\subset L^1_{\Loc}(\R^d)$ and that for any $q\in\mathcal P^+$ the annihilator of $\mathcal O(q)$ in $\mathcal L(\mathcal P)$ is trivial. In the light of Proposition \ref{prop: L1}, $\mathcal P^+$ is nowhere dense in $L^1(\R^d)$ as $\mathcal P^+\subset L^1_+(\R^d)$.

We now formally define {\Hy } entropy as the map from $\mathcal P^+$ to $\R$ given in \eqref{eq: Fi}. Convexity of $\Phi$ follows from the convexity of the function
\[
\phi(t,t_1,\dots,d_d)=\frac{t_1^2+\cdots+t_d^2} t,\quad \text{ for } t>0,\,(t_1,\dots,t_d)\in\R^d,
\]
while its 1-homogeneity is trivial. Hence, $\Phi$ is sublinear. Let us compute its right directional derivative.

For $q\in\mathcal P^+$ and $p\in\mathcal D(q)$, we set $q_t=q+tp$. We have
\begin{align*}
  \lim_{t\To0^+} \frac{\Phi(q+tp) - \Phi(q)}{t} & = \int_{\R^d} \frac d {dt}\Bigg|_{t=0} \left(\frac  {\abs{\nabla q_t(x)}^2} {q_t(x)} \right)dx\\
                                            & = \int_{\R^d}\left( 2\frac{ \nabla q(x)} {q(x)} \frac {\nabla p(x)}{p(x)} - \frac{\abs{\nabla q(x)}^2}{ q^2(x)}\right)p(x)dx.
\end{align*}
By integration by parts we get
\begin{multline*}
  \int_{\abs x \leq R}\left( \frac{2 \nabla q(x){\nabla p}(x)} {q(x)}  - \frac{\abs{\nabla q(x)}^2}{ q^2(x)}{{ p}}(x)\right)dx\\
                = \int_{\abs x \leq R} \left(-\frac{2 \Delta q(x)}{q(x)} + \frac{\abs{ \nabla q(x)}^2}{ q^2(x)}\right) p(x)dx + \frac 2 R \int_{\abs y = R} \left( \frac{ y \nabla q(y)}{q(y)}  \right) {p}(y)dy.
\end{multline*}
Letting $R\To\infty$ and using \eqref{eq: lim 1}, we obtain
\begin{equation*}
  \Phi_+'(p,q) = \int_{\R^d} \left(-\frac{2 \Delta q(x)}{q(x)} + \frac{\abs{ \nabla q(x)}^2}{ q^2(x)}\right) p(x)dx.
\end{equation*}
The assumptions on $\mathcal P^+$ guarantee that
\[
   S(q)(x) = -\frac{2 \Delta q(x)}{q(x)} + \frac{\abs{ \nabla q(x)}^2}{ q^2(x)}
\]
is $\mathcal P$-integrable for every $q\in\mathcal P^+$. In view of Theorem \ref{thm: uniq}, $S(q)$ is the unique $\mathcal P$-integrable subgradient of $\Phi$ on $\mathcal P^+$. The map is known as the \emph{H\"yvarinen scoring rule} \citep{PDL}.

In fact, $S(q)$ is a strict subgradient of $\Phi$ on $\mathcal P^+$. This can be shown if we notice that the divergence induced by $S$ has the representation
\begin{align*}
  p \cdot S(p) - p \cdot S(q) = \int_{\R^d} \abs{\frac {\nabla p(x)}{p(x)} - \frac {\nabla q(x)}{q(x)}}^2 p(x)dx.
\end{align*}
The latter identity can be proved by integration by parts. The divergence is zero if and only if
\[
   \nabla (\ln p(x) - \ln q(x))=0.
\]
This is equivalent to $p=Cq$ for some constant $C>0$, i.e., $p$ and $q$ being positively collinear. This concludes the proof of the claim.

\subsection{Quadratic Entropy}\label{sect: Qua}

Here we consider the \emph{quadratic entropy}
\begin{equation}\label{eq: qdr}
  \Phi(q) = \frac 1 {q \cdot 1} \int_{\Omega} q^2(x)dx,
\end{equation}
where $(\Omega,\mathcal A, \mu)$ is a Lebesgue measure space with $\Omega\subset\R^d$. In what follows, we show that its G\^ateaux derivative is the \emph{quadratic scoring rule}, also known as \emph{Brier score}. The quadratic entropy is a member of the important family of \emph{power entropy functions}. The corresponding \emph{power scoring rules} have been studied in connection to robust estimation e.g. in \cite{BHHJ, KF, KF2}. 

We proceed to choose a suitable domain for $\Phi$. In contrast to the previous two entropies we now introduce a topology. To that end, we begin with a description of some normed spaces. Let $w:\Omega\To[0,\infty)$ be a measurable function which we call a \emph{weight}. By $L^p(\Omega,w)$, for $p\geq1$, we denote the Lebesgue space of functions on $\Omega$ whose $p$-th power is absolutely integrable with respect to the weight $w(x)$. By $\norm \cdot_{p,w}$ we denote the corresponding weighted $L^p$-norm. When $w$ is identically equal to one we get the usual Lebesgue space and norm. In this case we drop $w$ from our notation. We now set
\[
   w(x) = (1 + \abs x)^{ d + 1}.
\]
Notice that $L^2(\Omega,w)$ embeds continuously in $L^1(\Omega)$. Indeed, for $f\in L^1(\Omega)$, we have
\begin{align*}
 \int_\Omega \abs{f(x)}dx  &=  \int_\Omega w^{-1/2}(x)\abs{f(x)}w^{1/2}(x)dx \\
                           &\leq  \left(\int_\Omega w^{-1}(x)dx\right)^{\frac 1 2} \left(\int_\Omega \abs{f(x)}^2w(x)dx\right)^{\frac 1 2}\\
                           &\leq C \norm{f}_{2,w},
\end{align*}
where $C>0$ is a constant. Clearly, $L^2(\Omega,w)$ also embeds continuously in $L^2(\Omega)$ and hence the same conclusion holds for $L^2(\Omega,w)$ for all intermediate  spaces $L^p(\Omega)$ with $1\leq p\leq 2$. Hence, we have the inequality
\begin{equation*}
  \norm{f}_p \leq C \norm{f}_{2,w}
\end{equation*}
for some fixed $C>0$ and all $p\in[1,2]$.

We have that $f\in L^2(\Omega,w)$ if and only if $fw^{1/2}\in L^2(\Omega)$. Clearly, the weight is needed only when $\Omega$ is unbounded as otherwise the weighted and the ordinary $L^p$-norms are equivalent. The continuous dual space of $L^2(\Omega,w)$ may be identified with the space $L^2(\Omega,w^{-1})$. Therefore, $g\in L^2(\Omega,w^{-1})$ if and only if $gw^{-1/2}\in L^2(\Omega)$. Hence, the dual space $L^2(\Omega,w^{-1})$ contains the constants and also the elements of $L^2(\Omega,w)$.

We now specify a prediction set $\mathcal P\subset L_+^2(\Omega,w)$ with the following property: there are constants $k_1>0$ and $k_2>0$ such that
\begin{equation*}
  k_1 \leq \norm q_{2,w} \leq k_2
\end{equation*}
for all $q\in\mathcal P$. Choose $0<\epsilon<\min(1,k_1)$. For $p\in L^2(\Omega)$, let $B_\rho(p)$ denote the open ball  about $p$ of radius $\rho>0$. Choose $\delta>0$ so small that for every $p\in B_\delta(0)$ we have
$\norm p_1 \leq \epsilon$ and $\norm p_{2,w} \leq \epsilon$. Let $q\in\mathcal P$ and consider $r\in B_\delta(q)$. It is easy to show that
\begin{equation*}
  k_1 - \epsilon \leq \norm r_{2,w} \leq k_2 + \epsilon
\end{equation*}
for all $r\in B_\delta(q)$. Similarly, we also have
\begin{equation*}
  1 - \epsilon \leq r \cdot 1 \leq 1 + \epsilon
\end{equation*}
for all $r\in B_\delta(q)$. Here we have used the fact that $r=p+q$, where $q\cdot 1 = 1$ and $\norm p_1\leq \epsilon$. We now set
\[
    \mathcal C_0 = \mathcal P + B_\delta(0) = \cup_{q\in\mathcal P} B_\delta(q).
\]
It follows that $\mathcal C_0$ is convex as both $\mathcal P$ and $B_\delta(0)$ are convex. Finally, let $\mathcal C=\mathcal C_0^+$ be the cone of $\mathcal C_0$. Clearly, $\mathcal C$ is an open convex cone in $L^2(\Omega,w)$.

We may now formally define $\Phi$ as the map from $\mathcal C$ to $\R$ given by \eqref{eq: qdr}. We have that $\Phi$ is strictly convex on $\mathcal C_0$ as the kernel function $\phi(t)=t^2$ is strictly convex for $t\in\R$. Therefore, $\Phi$ is strictly sublinear on $\mathcal C$. It is not hard to see that $\Phi$ is also continuous on $\mathcal C$. Theorem \ref{thm: N1} implies that $\Phi$ has a subgradient on $\mathcal C$. The following computation shows that $\Phi$ is G\^ateaux differentiable. Indeed, for $q\in\mathcal C$ and $p\in L^2(\Omega,w)$, we have
\begin{align*}
  \lim_{t\To 0} \frac{\Phi( q + t  p)-\Phi( q)}{t} = \int_\Omega \frac{d}{dt}\Bigg|_{t=0}\frac{{( q(x)+t{ p}(x))^2}}{( q+tp)\cdot 1}dx\\
                = 2\int_\Omega \frac{ q(x) { p}(x)}{q\cdot 1} dx - \int_\Omega \frac{ q^2(x)}{(q\cdot 1)^2}dx\int_{\Omega}{ p}(x)dx.
\end{align*}
We obtain that
\begin{equation*}
  \nabla\Phi(q) =\frac {2q} {q\cdot 1}  - \frac {q\cdot q}{(q\cdot 1)^2}
\end{equation*}
is the G\^ateaux derivative of $\Phi$ as clearly $\nabla\Phi(q)\in L^2(\Omega,w^{-1})$.
In view of Theorem \ref{thm: N2}, $S=\nabla\Phi|_{\mathcal P^+}$ defines a strictly proper scoring rule on $\mathcal P^+$. We have that $\nabla\Phi$ is the unique subgradient of quadratic entropy on the cone $\mathcal C$, but as discussed before, by using the Hahn-Banach theorem one may show that uniqueness fails on $\mathcal P^+$ when $\Omega$ is unbounded. The rule $S$ is known as the \emph{quadratic scoring rule}.

\section{Conclusion}\label{sect: conc}

We were originally motivated to understand the implications of the fact that Shannon and {\Hy } entropies are only finite on domains with empty interiors. As no notion of functional derivative is applicable to these entropies, the question whether the logarithmic and {\Hy } scoring rules are the unique subgradients of their respective entropy functions is not obvious. In contrast, the quadratic entropy may be continuously extended to signed densities, which allows us to interpret the quadratic scoring rule as the G\^ateaux derivative of its entropy. We realised that in order to answer the titular question of the paper, one must introduce additional structures to the basic measure-theoretic framework known in the literature of scoring rules \citep{HB}. The most important new aspect is the notion of interior and its refinement (known as quasi-interior) in the context of domains with empty interior. Another crucially important idea is to use directional derivatives to describe the subdifferentials of entropy functions. Finally, our approach marks a shift in emphasis from proper scoring rules to a greater focus on entropy functions.

\appendix
\counterwithin{thm}{section}

\section{Proofs}\label{app: proofs}

\begin{lem}\label{lem: subg-subl}
Let $\mathcal P$ be a prediction set and $\Phi:\mathcal P^+\To\R$ be a 1-homogeneous function. If $\Phi$ has a (strict) subgradient on $\mathcal P^+$, then $\Phi$ is a (strictly) sublinear function.
\end{lem}
\begin{proof}
Let $S:\mathcal P^+\To\Lin \mathcal P$ be a (strict) subgradient of $\Phi$. Then $S$ (strictly) satisfies
\begin{align*}
  \Phi( p) &\geq p\cdot S((1-\lambda) p + \lambda q) \\
  \Phi( q) &\geq q\cdot S( (1-\lambda) p + \lambda q)
\end{align*}
for every $p,q\in\mathcal P^+$ ($p$ and $q$ not positively collinear), and every $0<\lambda<1$. Multiplying the first inequality by $1-\lambda$, the second one by $\lambda$, and then adding them up, we obtain that $\Phi$ (strictly) satisfies
\[
  \Phi(1-\lambda) p + \lambda q) \leq (1-\lambda)\Phi( p) + \lambda\Phi( q). 
\]
\end{proof}

\begin{proof}[Proof of Lemma \ref{lem: Dir}]
We first show that $\Cone(\mathcal P^+ -  {q})\subset {\mathcal D}( {q})$. It is easy to see that ${\mathcal D}( {q})$ is closed under taking conic combinations. The claim follows from the fact that $(\mathcal P^+ -  {q})\subset {\mathcal D}( {q})$. We now show that ${\mathcal D}( {q})\subset\Cone(\mathcal P^+ -  {q})$. If $ {p}\in {\mathcal D}( {q})$, then there is $\epsilon_{ {p}}>0$ and $ {r}\in \mathcal P^+$ such that $ {q}+\epsilon_{ {p}} {p}= {r}$. Then $ {p}=( {r}- {q})\epsilon_{ {p}}^{-1}$ and hence $ {p}\in \Cone (\mathcal P^+ - {q})$.
\end{proof}

\begin{proof}[Proof of Proposition \ref{prop: Phi+}]
(a) For $p\in\mathcal D(q)$ arbitrary, consider the line in $\Span \mathcal P$ with parametric equation
\[
    \gamma(t)=q + t(p-q), \quad t\in\R,
\]
passing through $q$ and $p$. Clearly, $\gamma(0)=q$ and $\gamma(1)=p$. Moreover, there is some $\epsilon>0$ such that the interval $[0,\epsilon]$ is mapped entirely in $\mathcal P^+$ under $\gamma$ (if $p\in\mathcal P^+$, then $\epsilon\geq1$). Then the function
\[
   \phi(t) = \Phi(q + t(p-q)), \quad t\in [0,{\epsilon}],
\]
is convex and its slope function
\[
   s_\phi(t_1,t_2) = \frac {\phi(t_2)-\phi(t_1)}{t_2-t_1},\quad t_1,t_2\in  [0,{\epsilon}],
\]
is nondecreasing \citep{Rock, NicPer}. We have that
\[
  \Phi_+'(p,q) = \lim_{t_2\To0+} \frac {\phi(t_2)-\phi(0)}{t_2} = \inf_{t_2>0} \frac {\phi(t_2)-\phi(0)}{t_2}.
\]

If $p\in\mathcal O(q)$, then there is some $\delta>0$ such that the interval $[-\delta,\delta]$ is mapped entirely in $\mathcal P^+$ under $\gamma$. Let $-\delta\leq t_1<0<t_2\leq \delta$. To prove that $\Phi_+'(p,q)$ is finite, we consider
\[
   \frac {\phi(0)-\phi(t_1)}{-t_1} \leq \frac {\phi(t_2)-\phi(0)}{t_2},
\]
and take the infimum in $t_2$.

(b) Homogeneity of $\Phi_+'(\cdot, q)$ follows from:
\begin{align*}
  \Phi_+(\lambda  p, q) =  \lim_{\tau\To 0^+} \frac{\Phi( q +  \tau \lambda  p) - \Phi( q)}{\tau}
                                         &\leq \lambda \lim_{\tau\To 0^+} \frac{\Phi( q + \lambda \tau  p) - \Phi( q)}{\lambda \tau}\\ &= \lambda \Phi_+( p, q).
\end{align*}
Let $ p_1,  p_2\in\mathcal D(q)$. Subadditivity of $\Phi_+'(\cdot, q)$ follows from:
\begin{align*}
  \Phi_+'( p_1+ p_2, q) &= \lim_{\tau\To 0^+} \frac{\Phi( q + \tau ( p_1+ p_2)) - \Phi( q)}{\tau} \\
                                         &\leq \lim_{\tau\To 0^+} \frac{\Phi( q/2 + \tau  p_1) - \Phi( q)/2}{\tau} + \lim_{\tau\To 0^+} \frac{\Phi( q/2 + \tau  p_2) - \Phi( q)/2}{\tau}\\
                                         &=\lim_{\tau\To 0^+} \frac{\Phi( q + 2\tau  p_1) - \Phi( q)}{2\tau} + \lim_{\tau\To 0^+} \frac{\Phi( q + 2\tau  p_2) - \Phi( q)}{2\tau}\\
                                         &= \Phi_+'( p_1, q) + \Phi_+'( p_2, q).
\end{align*}

(c) The claim follows from
\begin{align*}
  \Phi_+'(p, \lambda q) &=  \lim_{\tau\To 0^+} \frac{\Phi( \lambda q +  \tau  p) - \Phi(\lambda q)}{\tau}
                        =  \lim_{\tau\To 0^+} \frac{\Phi( q +  \tau  p/\lambda) - \Phi( q)}{\tau/\lambda}\\
                       &=  \Phi_+'( p, q).
\end{align*}

(d) We have
\[
   \Phi( p) \geq \Phi( q + p)-\Phi( q) \geq \frac{\Phi( q + \tau p)-\Phi( q)}{\tau} \geq \Phi_+'( p,  q),
\]
where $0<\tau<1$. The first inequality follows from sublinearity of $\Phi$, while the second and third follow from the fact that the slope function of $\Phi$ is nondecreasing. It remains to show that $\Phi( q)=\Phi_+'( q,  q)$. This follows immediately from
\begin{align*}
  \Phi({ q}) &= \lim_{\tau\To 0^+}\frac {(1+\tau)\Phi({ q})- \Phi({ q})}{\tau}
              =  \lim_{\tau\To 0^+}\frac {\Phi({ q} + \tau { q})- \Phi({ q})}{\tau}\\
             &= \Phi_+'( q,  q).
\end{align*}

(e) The claim is a direct consequence of
\[
   0 = \Phi_+'(0,q) = \Phi_+'(p-p,q) \leq \Phi_+'(p,q) + \Phi_+'(-p,q).
\]

(f) To show that $\mathcal O'(q)$ is a linear subspace of $\mathcal O(q)$ it is enough to show that it is closed under scalar multiplication and vector addition. Let $\lambda\in \R$ and $p\in \mathcal O'(q)$. Then, for $\lambda\geq 0$, $\Phi_+'(\lambda p,q)=\lambda \Phi_+'( p,q)$. Analogously, for $\lambda<0$ we have
\[
  \Phi_+'(\lambda p,q)=\Phi_+'(-\lambda (-p),q)=-\lambda \Phi_+'(-p,q)=\lambda (- \Phi_+'(-p,q))=\lambda \Phi_+'(p,q).
\]
Therefore, $\Phi_+'(\lambda p,q)=\lambda \Phi_+'(p,q)$ for any $\lambda\in\R$ and $p\in \mathcal O'(q)$. Then multiplying by $\lambda$ both sides of the identity
\[
  -\Phi_+'(-p,q) = \Phi_+'(p,q)
\]
and using the previous identity, we get that $\lambda p \in \mathcal O'(q)$. Hence, $\mathcal O'(q)$ is closed under scalar multiplication.

Suppose now that $p,r\in \mathcal O'(q)$. We have
\begin{align*}
  \Phi_+'(p+r,q) &\leq \Phi_+'(p,q) + \Phi_+'(r,q) = -(\Phi_+'(-p,q) + \Phi_+'(-r,q))\\
                 &\leq -\Phi_+'(-p-r,q) \leq \Phi_+'(p+r,q),
\end{align*}
where the last inequality follows from (e). Clearly, we must have equalities throughout. In particular,
\[
  -\Phi_+'(-p-r,q) = \Phi_+'(p+r,q)
\]
and
\begin{align*}
  \Phi_+'(p+r,q) = \Phi_+'(p,q) + \Phi_+'(r,q).
\end{align*}
Hence $p+r\in \mathcal O'(q)$. We conclude that $\mathcal O'(q)$ is a linear subspace and $\Phi_+'(\cdot,q)\big|_{\mathcal O'(q)}$ is linear.
\end{proof}

\begin{proof}[Proof of Proposition \ref{prop: subdiff}]
(a) The sufficient part of the claim follows from Proposition \ref{prop: Phi+} (d). Let us now show the necessary part. To that end, let $q^*\in \Lin \mathcal P$ be a subgradient of $\Phi$ at $q$, and let $p\in\mathcal P^+$ be arbitrary. Setting $q_t=q+(1-t)p$, we have $\Phi(q_t)\geq q_t\cdot q^*$ for all $t\in[0,1]$. Subtracting $\Phi(q)$ from both sides of the inequality and dividing by $(1-t)$, for $t\in(0,1)$, we get
\[
   \frac{\Phi(q+(1-t)p) - \Phi(q)} {1-t} \geq p\cdot q^*.
\]
Letting $t\uparrow 1$, we get
\[
   \Phi_+'(p,q)\geq p\cdot q^*
\]
as desired.

(b) The claim follows by restricting $\Phi$ to 1-dimensional affine spaces through $q$. On these spaces $\Phi$ is convex and differentiable and therefore has a unique subgradient. Since these subspaces cover the whole of $\Span\mathcal P$, it follows that the directional derivative $\Phi'(\cdot,q)$ is the unique subgradient of $\Phi$ there.

(c) In view of Proposition \ref{prop: Phi+} (a), $\Phi_+'(p,q)$ is finite for each $p\in\mathcal O(q)=\Span\mathcal P$. The hypothesis implies that there is at least one 1-dimensional linear subspace of $\Span \mathcal P$ on which $\Phi_+'(\cdot,q)$ is not linear. There are infinitely many ways we can choose a linear function on that space that is dominated by $\Phi_+'(\cdot,q)$. The claim now follows from the Hahn-Banach theorem stated below as Theorem \ref{thm: Hahn-B}.

(d) Since $\mathcal O(q)\not=\Span\mathcal P$, it follows that $\mathcal P^+\setminus \mathcal O(q)$ is nonempty. Take any $p$ in that set and consider the 1-dimensional linear space generated by the span of $p$. Since $\Phi_+'(\cdot,q)$ is defined only on its positive half-space, there are infinitely many linear functions that are dominated by $\Phi_+'(\cdot,q)$ on the whole space. The proof now follows from Theorem \ref{thm: Hahn-B}.

(e) There is no element of $\Lin \mathcal P$ that satisfies the condition in part (a) of this proposition. Therefore, $\partial\Phi(q)=\emptyset$.
\end{proof}

\begin{proof}[Proof of Theorem \ref{thm: exist}]
Suppose that $q^*\in \mathcal L(\mathcal P)$ satisfies $p\cdot q^*\leq\Phi_+'(p,q)$ for all $p\in\mathcal P^+$, with equality for $p=q$. In view of Proposition \ref{prop: Phi+} (d), we have that $p\cdot q^*\leq \Phi(p)$ for all $p\in\mathcal P^+$, and $q\cdot q^* = \Phi(q)$. Hence, $q^*$ is a $\mathcal P$-integrable subgradient of $\Phi$ at $q$.

The converse claim, that is, if  $q^*$ is a $\mathcal P$-integrable subgradient of $\Phi$ at $q$, then $p\cdot q^*\leq\Phi_+'(p,q)$ for all $p\in\mathcal P^+$, with equality for $p=q$, follows from Proposition \ref{prop: subdiff} (a).
\end{proof}

\begin{proof}[Proof of Theorem \ref{thm: uniq}]
The hypothesis implies that $\Phi_+'(\cdot,q)$ is linear on $\mathcal O(q)\subset \mathcal P^+$. By restricting $\Phi$ to 1-dimensional subspaces of $\mathcal O(q)$ it follows immediately that any subgradient of $\Phi$ must agree with $q^*$ on $\mathcal O(q)$. The assumption that $\mathcal O(q)^\perp=\{0\}$ implies that $\Phi$ may have at most one $\mathcal P$-integrable subgradient at $q$. Then the claim follows from the fact that $q^*$ is a subgradient of $\Phi$ at $q$.
\end{proof}

\section{Some Additional Facts}\label{app: exmpl}

The positive cones in many standard function spaces are nowhere dense sets. Let us show this for the Lebesgue space $L^1(\R^d)$. The positive cone of $L^1(\R^d)$ consists of all Lebesgue integrable functions $f\geq0$ a.e. on $\R^d$ and is denoted by $L^1_+(\R^d)$. We recall that a set in a topological vector space is \emph{nowhere dense} if its closure has empty interior.

\begin{prop}\label{prop: L1}
The positive cone of $L^1(\R^d)$ is nowhere dense.
\end{prop}

\begin{proof}
We show that for every $f\geq 0$ a.e., there is $g\geq 0$ a.e. such that, for every $\alpha>0$, $f-\alpha g\not\in L^1_+(\Omega)$. This means that no open ball about $f$ is contained in $L^1_+(\R^d)$. Since $L^1_+(\R^d)$ is closed, then this would imply that $L^1_+(\R^d)$ is nowhere dense.

To prove our claim, we use the fact that there is no absolutely convergent series with a slowest rate of decay at infinity. We begin by partitioning $\R^d$ into dyadic regions
\[
  \omega_k = \{2^k \leq |x| < 2^{k+1}\}
\]
for $k \in \Z$. For $f\in L^1(\R^d)$, we set
\[
  a_k = \int_{\omega_k} f(x)dx.
\]
We have that the series
\[
   \sum_{k=0}^\infty a_k = \int_{\R^d} f(x)dx
\]
is absolutely convergent. If $r_k=\sum_{i\geq k}a_i$ is the tail of the series for each $k$, then the series $\sum_{k\geq0}a_k/\sqrt{r_k}$ is also convergent \citep{Rud1}. Notice that the ratio of the common term of the second to the first series tends to infinity as $k\To\infty$. Therefore, the second series has a strictly slower rate of convergence. There exists a function $g\in L^1_+(\R^d)$ such that the integrals of $g$ on $\omega_k$ are $b_k=a_k /\sqrt{r_k}$ and
\[
   \sum_{k=0}^\infty b_k = \int_{\R^d} g(x)dx.
\]
Clearly, for any $\alpha>0$, the difference $f-\alpha g$ changes sign for some $\omega_k$, and hence $f-\alpha g\not\in L^1_+(\R^d)$.
\end{proof}

The next example illustrates the notion of topological subgradient in the case when the assumption
$N^*\subset\mathcal L(\mathcal P)$ is not satisfied.

\begin{exmpl}\label{exmpl: sup}
Consider a  Lebesgue measure space  $(\Omega, \mathcal A, \mu)$ with $\Omega$ a compact subset of $\R^d$. We set $\mathcal P^+$ to be the positive cone of $C(\Omega)$, that is, the set of all nonnegative continuous functions on $\Omega$. The continuous dual of $C(\Omega)$ is the space of all real-valued Radon measures on $\Omega$. The fact that $\mathcal P^+$ contains constants implies that $\mathcal L(\mathcal P)\subseteq L^1(\Omega)$. Actually, $\mathcal L(\mathcal P)=L^1(\Omega)$ and hence the $\mathcal P$-integrable functions are the Radon measures that have a Lebesgue density. Since $L^1(\Omega)\subsetneq (C(\Omega))^*$, we see that in this case the notion of a $\mathcal P$-integrable subgradient is more restrictive than that of a topological subgradient.

We proceed to examine the implications of the latter observation on a concrete sublinear function. Let $\Phi: C(\Omega)\To\R$ be the \emph{supremum function}, that is,
\[
  \Phi(p) = \sup_{x\in\Omega} p(x).
\]
It is easy to check that $\Phi$ is non-strictly sublinear and continuous. The supporting hyperplane theorem guarantees the existence of a topological subgradient of $\Phi$ at each point in its domain that is a \emph{real Radon measure}. Let us see whether the subgradient is regular enough to be a proper scoring rule.

We first demonstrate that there are points $q\in\mathcal P^+$ at which $\Phi$ has no subgradient in $\mathcal L(\mathcal P)$. To that end, let $\mathcal M(q)$ denote the set of \emph{modes} of $q$, that is, the subset of $\Omega$ where $q$ reaches its maximum. Notice that $\mathcal M(q)$  is always compact. It can be shown that
\[
   {\Phi_+'}({ p}, q) = \sup_{x\in \mathcal M( q)} { p}(x),
\]
the proof of which is left to the reader. When $\mathcal M(q)=\{x_0\}$ is a singleton,  $\Phi_+'(\cdot, q)=\delta(x-x_0)$ is Dirac's delta function. Clearly, in this case $\Phi$ is G\^ateaux differentiable with derivative $\delta(x-x_0)$. We claim that $\Phi$ has no $\mathcal P$-integrable subgradient for any density $q$ with $\mu(\mathcal M(q))=0$.

Suppose conversely that $q^*\in\mathcal L(\mathcal P)$, $q^*\not=0$, is a subgradient of $\Phi$ at $q$. Then
\[
   \Phi_+'(p, q) \geq p \cdot q^*
\]
for all $p\in\mathcal P^+$. We shall show that this inequality implies $q^*(x)\leq0$ a.e. on $\Omega$, which leads to a contradiction with $\Phi(q)=q\cdot q^*>0$.

To show the latter claim, notice that $\Omega\setminus \mathcal M(q)$ is open, and hence for any $y\in \Omega\setminus \mathcal M(q)$, there is $\epsilon_y>0$ such that the ball about $y$ of radius $\epsilon_y$ lies in the complement of $\mathcal M(q)$ with respect to $\Omega$. Let $\{p_k\}$ be a sequence of densities approximating $\delta(x-y)$ entirely supported on this ball. Since $\Phi_+'(p_k,q)=0$, we get that $p_k \cdot q^*\leq 0$. If $y$ is a Lebesgue point of $q^*$, then we have the limit
\[
   \lim_{k\To\infty} p_k \cdot q^* = \delta (\cdot-y)\cdot q^* = q^*(y).
\]
Since almost every point of $q^*$ is a Lebesgue point, we get that $q^*(x)\leq 0$ a.e. on $\Omega$. This completes the proof of the claim.

In the case $\mu(\mathcal M(q))>0$, we may find a $\mathcal P$-integrable subgradient of $\Phi$ at $q$. Consider the function
\begin{align*}
  q^*(x)       = \begin{cases}
                        \frac{1}{\mu(\mathcal M(q))}       & x\in \mathcal M(q)\\
                        0                                  & x\in \Omega\setminus \mathcal M(q).
                  \end{cases}
\end{align*}
Clearly, $q\cdot q^*=\sup_{x\in\Omega} q(x)$ and $p\cdot q^* \leq \sup_{x\in\Omega} p(x)$ for all $p\in\mathcal P^+$. This furnishes our claim.
\end{exmpl}

In our final example, we illustrate the fact that at boundary points a sublinear function has either no subgradient, or infinitely many.

\begin{exmpl}\label{exmpl: Hahn-B}
Take $\Phi(x,y)=x+y$ on $\R_+^2=\{(x,y)\,|\, x\geq 0,\,y\geq0\}$. The graph of $\Phi$ is a part of a plane, so it is easy to see that $\Phi$ has infinitely many supporting planes at the boundaries of $\R_+^2$. Consider now
\[
   \Phi(x,y)=x\ln \frac x {x+y} + y\ln \frac y {x+y}
\]
on $\R_+^2$, which is Shannon entropy for binary variables. A computation shows that
\[
   \nabla\Phi(x,y) = \ln \frac x {x+y} + \ln \frac y {x+y}
\]
and hence $\nabla\Phi(x,y)\To-\infty$ when $(x,y)$ tends to the boundary of $\R_+^2$. This means that $\Phi$ has vertical tangent planes through the coordinate axes, which implies that $\Phi$ has no subgradient on the boundary of its domain.

The situation is the same when $\mathcal P^+$ is a subset of an infinite dimensional vector space. For example, one may use the Hahn-Banach theorem presented below to show the existence of multiple supporting hyperplanes at boundary points $q$ for which $\Phi_+'(p,q)$ is finite for all $p\in\mathcal P^+$. If, instead, there is $p\in\mathcal P^+$ for which $\Phi_+'(p,q)=-\infty$, then $\Phi$ has no subgradient at $q$.
\end{exmpl}

We now state a slight generalisation of the classical Hahn-Banach theorem. Let $E$ be a real vector space and $K\subset E$ be a convex cone.

\begin{thm}[Hahn-Banach theorem]\label{thm: Hahn-B}
Let $\phi : K \To \R$ be a sublinear function and $l_0 : E_0 \To \R$ be a linear functional on a linear subspace $E_0 \subseteq E$ which is dominated by $\phi$ on $E_0\cap K$, i.e.
\begin{align*}
  l_0(q) \leq \phi(q), \qquad\forall q \in E_0\cap K.
\end{align*}
Then there exists a linear extension $l : E \To \R$ of $l_0$ to the whole space $E$ such that
\begin{align*}
 l(q)  &= l_0(q),     \qquad\forall q\in E_0,\\
 l(q)  &\le \phi(q),  \qquad\forall q\in E\cap K.
\end{align*}
\end{thm}

In the classical formulation of the theorem, we have $K=E$. The proof of the version with $K\subset E$ is the same. In fact, if anything, the condition $K\subset E$ is easier to satisfy than $K=E$ when extending $l_0$.

\bibliographystyle{plainnat}
\bibliography{PSR-database}

\end{document}